\documentclass[11pt,a4paper]{amsart}
\usepackage{amssymb}
\usepackage{enumerate}
\usepackage{graphicx}

\usepackage{amsmath}
\usepackage{amsfonts}
\usepackage{amsthm}
\usepackage{color}

 \newtheorem{teo}{Theorem}
 \newtheorem{lema}{Lemma}
 \newtheorem{coro}{Corollary}

 \newtheorem{prop}{Proposition}

\newcommand{\ZZ}{\mathbb{Z}}
\newcommand{\RR}{\mathbb{R}}

\newcommand{\pp}{\mathbf{p}}
\newcommand{\os}{\mathbf{s}}
\newcommand{\QQ}{\mathcal Q}

\begin{document}

\title[Limit cycles for a class of $\mathbb{Z}_{2n}-$equivariant equations]
{Limit cycles for a class of $\mathbb{Z}_{2n}-$equivariant systems without infinite equilibria}
\author{Isabel S. Labouriau}
\address{I.S. Labouriau, Centro de Matem\'atica da Universidade do Porto.\\ Rua do Campo Alegre 687, 4169-007 Porto, Portugal}
\email{islabour@fc.up.pt}
\author{Adrian C. Murza}
\address{A.C. Murza, Centro de Matem\'atica da Universidade do Porto.\\ Rua do Campo Alegre 687, 4169-007 Porto, Portugal}
\email{adrian.murza@fc.up.pt}

\thanks{CMUP (UID/MAT/00144/2013) is funded by FCT (Portugal) with national (MEC) and European structural funds through the programs FEDER, under the partnership agreement PT2020.
A.C.M. acknowledges support  by the grant SFRH/ BD/ 64374/ 2009 of FCT
}

\begin{abstract}

We analyze the dynamics of a class of $\mathbb{Z}_{2n}$-equivariant differential equations of the form $\dot{z}=pz^{n-1}\bar{z}^{n-2}+sz^{n}\bar{z}^{n-1}-\bar{z}^{2n-1},$ where $z$ is
complex, the time $t$ is real, while $p$ and $s$
 are complex parameters. This study is the generalisation to $\mathbb{Z}_{2n}$  of  previous works with
 $\mathbb{Z}_4$  and  $\mathbb{Z}_6$ symmetry.
We  reduce the problem of finding limit cycles  to an Abel equation, and provide criteria for proving in some cases uniqueness and hyperbolicity of the limit cycle that surrounds either 1, $2n+1$ or $4n+1$ equilibria, the origin being always one of these points.
\end{abstract}

\maketitle

\bigbreak

\section{Introduction and main results}\label{Introduction and main results}
 Hilbert $XVI^{th}$ problem was the motivation for a large amount of articles over the last century, and  remains one of the open questions in mathematics.
The study of this problem in the context of equivariant dynamical systems is a
new branch of analysis, based on the development of equivariant bifurcation theory, by Golubitsky, Stewart and Schaeffer, \cite{GS85,GS88}.
Many other authors, for example Chow and Wang \cite{Che}, have considered this theory when studying the limit cycles and
related phenomena in systems with symmetry.

In this paper we analyze the $\mathbb{Z}_{2n}-$equivariant system
\begin{equation}\label{main equation}
\dot{z}=
\displaystyle{\frac{dz}{dt}}=\pp z^{n-1}\bar{z}^{n-2}+\os z^{n}\bar{z}^{n-1}-\bar{z}^{2n-1}
 =f(z)  ,
\end{equation}
for $n>3$, where  $\pp=p_1+ip_2$, $\os=s_1+is_2$, $p_1,p_2,s_1,s_2\in\RR$, $t\in\RR$.

The general form of the $\mathbb{Z}_q-$equivariant equation is
\[
 \dot{z}=zA(|z|^2)+B\bar{z}^{q-1}+O(|z|^{q+1}),
\]
where $A$ is a polynomial on the variable $|z|^2$ whose degree is the integer part of $(q-1)/2$.
This class of equations is studied, for instance in the books \cite{arn, Che}, when the resonances are
strong, {\it i.e.} $q<4$ or weak $q>4$.
A partial treatment of the  special case $q=4$ is given, for instance,  in the article \cite{Zegeling}, and in the book \cite{Che} that is 
concerned with normal forms and bifurcations in general.
A more complete treatment of the case  $q=4$ appears in the article \cite{Rafel1},
while the case $q=6$ appears in \cite{murza}.
All mentioned
articles claim the fact that, since the equivariant term $\bar{z}^ {q-1}$ is not dominant with respect to the function  on $\bar{z}^ 2,$ they are easier to study than other cases. While this argument works for obtaining the bifurcation diagram near the origin, it is no longer helpful for a global analysis or if the analysis is focused on the study of  limit cycles.
The aim of the present work is to study the global phase portrait of  \eqref{main equation} on the Poincar\'e compactifiction of the plane;  we devote especial interest to analysing the existence, location and uniqueness of limit cycles surrounding $1,$ $2n+1$ or $4n+1$ equilibria.
Our strategy
uses some of the techniques developed in \cite{Rafel1} and
 includes transforming
\eqref{main equation} into a scalar Abel equation followed by its analysis.

The main results of this article are Theorems~\ref{teoEquilibria} and \ref{teorema principal} below.
Consider the quadratic form
\begin{equation}\label{quadraticForm}
\QQ (p_1,p_2)=p_1^2+p_2^2-(p_1 s_2-p_2 s_1)^2=(1-s_2^2)p_1^2+(1-s_1^2)p_2^2+2s_1s_2p_1p_2 .
\end{equation}

\begin{teo}\label{teoEquilibria}
For $|s_2|>1$  and for any $s_1\ne 0$, $\pp\ne 0$,  if $p_2s_2\ge 0$ the only equilibrium of \eqref{main equation}  is the origin.
If $p_2s_2<0$ then the number of equilibria of  \eqref{main equation} is determined by the quadratic form
$\QQ (p_1,p_2)$ defined in \eqref{quadraticForm}
and is:
\begin{enumerate}
\item
exactly one equilibrium (the origin) if $\QQ (p_1,p_2)<0$;
\item
exactly $2n+1$ equilibria
 (the origin and one saddle-node  per  region $(k-1)\pi/n\le \theta<k\pi/n$, $k\in\ZZ$) if $\QQ (p_1,p_2)=0$;
\item
exactly $4n+1$ equilibria (the origin and two equilibria in each  region $(k-1)\pi/n\le \theta<k\pi/n$, $k\in\ZZ$) if $\QQ (p_1,p_2)>0$.
\end{enumerate}
\end{teo}

\begin{teo}\label{teorema principal}
For  $|s_2|>1$  and for any $s_1\ne 0$,   and $\pp\ne 0$, consider the conditions:
$$
(i)\quad \QQ(p_1,p_2)\le 0
\qquad
\qquad
(ii)\quad \QQ(2p_1,p_2)\le 0 .
$$

 $(a)$ If  either condition $(i)$ or $(ii)$ holds, then equation \eqref{main equation} has at most one limit cycle surrounding the origin, and when the limit cycle exists it is hyperbolic.

$(b)$ There are parameter values where $\QQ(p_1,p_2)< 0$ for which there is a stable limit cycle surrounding the origin.

$(c)$ There are parameter values where $\QQ(p_1,p_2)= 0$ for which there is a limit cycle surrounding the $2n+1$ equilibria given by Theorem~\ref{teoEquilibria}.

$(d)$ There are parameter values where $\QQ(2p_1,p_2)\le 0$ for which there is a limit cycle surrounding  either the $2n+1$ or the $4n+1$ equilibria given by Theorem~\ref{teoEquilibria}.
\end{teo}

\begin{figure}[ht]
\begin{center}
\includegraphics[scale=.5]{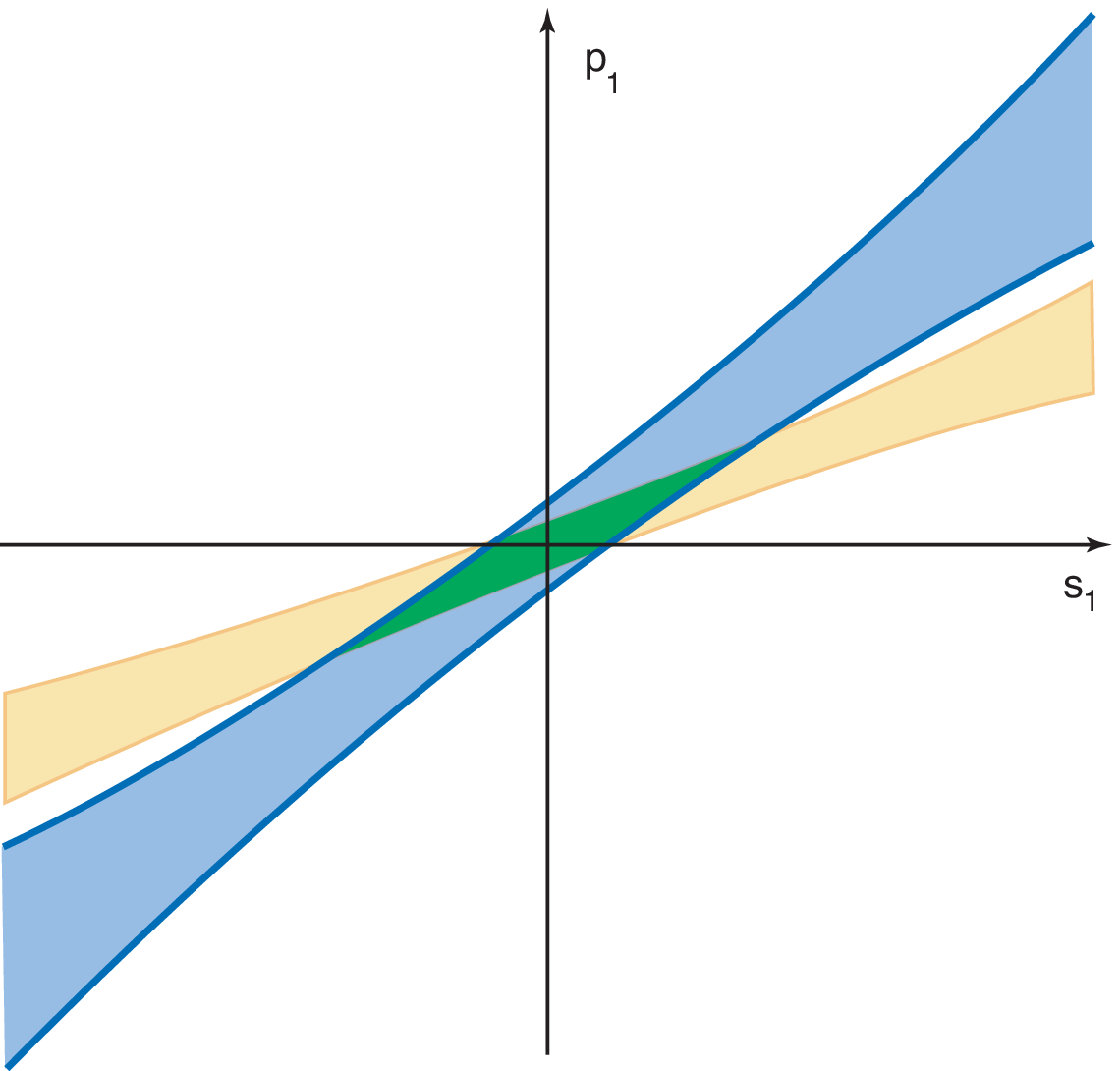}\\
\includegraphics[scale=.35]{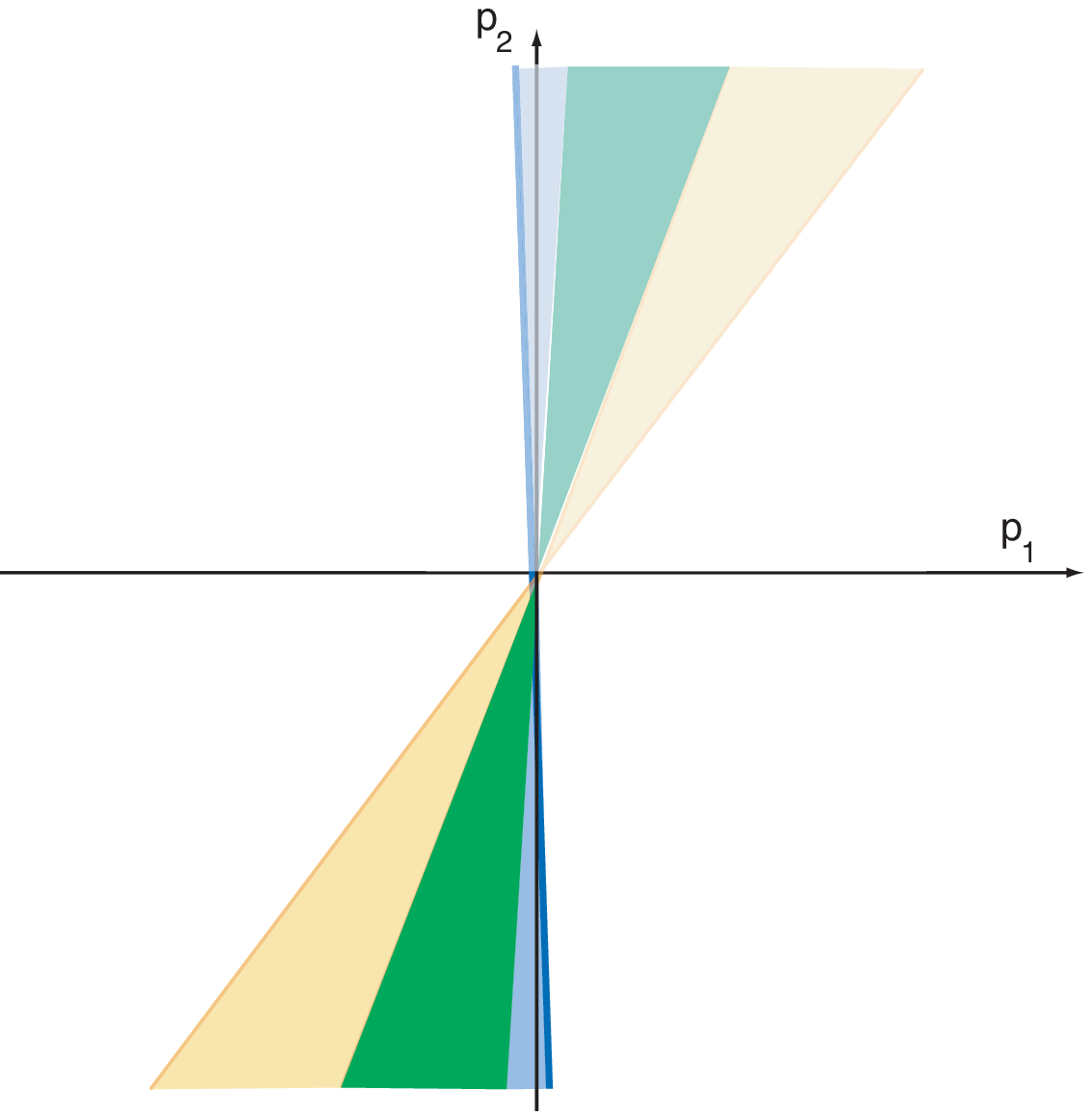}\quad
\includegraphics[scale=.35]{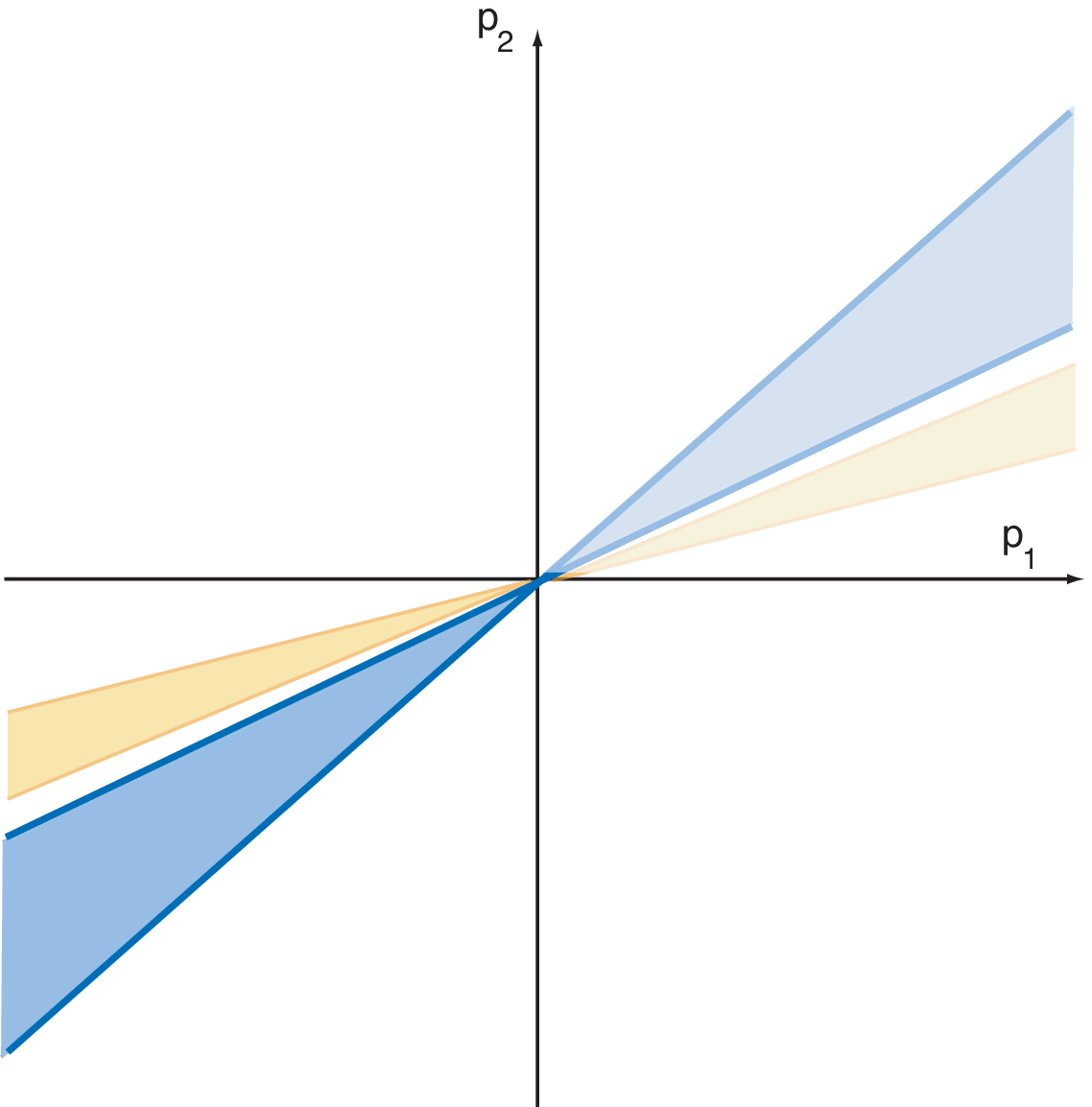}
\caption{The regions defined in Theorems~\ref{teoEquilibria} and \ref{teorema principal}: blue when $\QQ(p1,p_2)\ge 0$, yellow  when $\QQ(2p1,p_2)\ge 0$, green in the intersection of the two regions.
Top: diagram on the $(p_1,s_1)$-plane, with $p_2=1$ and $s_2=4$.
Bottom: diagrams on  the $(p_1,p_2)$-plane, with $s_1=1/2$, $s_2=4$ on the left, and with $s_1=6$, $s_2=4$ on the right.
There are $4n+1$ equilibria on the interior of the blue and green regions when $p_2s_2<0$ (darker colours).}\label{BD}
\end{center}
\end{figure}

This article is organised as follows.
After some preliminary results in  Section \ref{Preliminary results},
the number of equilibria is treated in Section \ref{Analysis of the equilibria}, as well as the proof of the Theorem \ref{teoEquilibria}.
The Abel equation is obtained in Section \ref{SecAbel}
and the proof of the Theorem \ref{teorema principal} is completed in
Section \ref{section proof}.
\section{Preliminary results}\label{Preliminary results}

Let $\Gamma$ be a closed subgroup of $\mathbf{O}(2)$.
A system of differential equations $dx/dt=f(x)$ in the plane is said to have symmetry $\Gamma$ (or to be $\Gamma$-equivariant) if $f(\gamma x)=\gamma f(x),~\forall\gamma\in\Gamma$.
Here we are concerned with $\Gamma=  \mathbb{Z}_{2n}$, acting on $\mathbb{C}\sim\RR^2$ by multiplication by
$\gamma_k=\exp(k\pi i/n)$, $k=0,1,\ldots,2n-1$.
For  equation \eqref{main equation} we obtain the following result.
\begin{prop}
Equation \eqref{main equation} is $\mathbb{Z}_{2n}-$equivariant.
\end{prop}
\begin{proof}
The  monomials in $z,~\bar{z}$ that appear in the expression of $f$ are $\bar{z}^{2n-1}$ and $z^{l+1}\bar{z}^l$.
The first of these is  $\gamma_k$-equivariant, while monomials of the form $z^{l+1}\bar{z}^l$  are $\mathbb{Z}_{2n}-$equivariant for all $n$.
\end{proof}

The next step is to identify the parameter values for which \eqref{main equation}  is Hamiltonian.

\begin{prop}\label{teo-Ham}
Equation \eqref{main equation} is Hamiltonian if and only if $p_1=0=s_1$.
\end{prop}
\begin{proof} The equation $\dot{z}=F(z,\bar{z})$ is Hamiltonian when $\frac{\partial F}{\partial z}+\frac{\partial \bar{F}}{\partial \bar{z}} =0$. For
equation \eqref{main equation} we have
$$
    \begin{array}{l}
    \displaystyle{\frac{\partial F}{\partial z}=(n-1)(p_1+ip_2)z^{n-2}\bar{z}^{n-2}+n(s_1+is_2)z^{n-1}\bar{z}^{n-1}}\\
    \\
    \displaystyle{\frac{\partial \bar{F}}{\partial \bar{z}}=(n-1)(p_1-ip_2)z^{n-2}\bar{z}^{n-2}+n(s_1-is_2)z^{n-1}\bar{z}^{n-1}}
    \end{array}
$$
and consequently it is Hamiltonian precisely when $p_1=s_1=0$.
\end{proof}

The expression of equation \eqref{main equation}  in polar coordinates will be useful.
Writing
$$
z=\sqrt{r}(\cos(\theta)+i\sin(\theta))
$$
and rescaling time as $\displaystyle{\frac{dt}{ds}}=r^{n-2}$,
we obtain
\begin{equation}\label{polar1}
\left\{
    \begin{array}{l}
    \displaystyle{\dot{r}=2r\left(p_1+rs_1-r\cos(2n\theta)\right)}\\
    \displaystyle{\dot{\theta}=p_2+rs_2+r\sin(2n\theta)}.
    \end{array}
    \right.
\end{equation}
The  symmetry means that for most of the time we only need to study the dynamics of  \eqref{polar1} in the fundamental domain for the $\mathbb{Z}_{2n}$-action, an angular sector of $\pi/n$.
It will often be convenient to look instead at the behaviour of a rescaled angular variable $\phi=n\theta$ in intervals of length
$\pi$, where the  equation  \eqref{polar1} takes the form
\begin{equation}\label{polarn}
\left\{
    \begin{array}{l}
    \displaystyle{\dot{r}=2r\left(p_1+rs_1-r\cos( 2\phi)\right)}\\
    \displaystyle{n\dot{\phi}=p_2+rs_2+r\sin(2\phi)}.
    \end{array}
    \right.
\end{equation}

One possible argument for existence of a limit cycle is to show that, in the Poincar\'e compactification, there are no critical
points at infinity and that  infinity and  the origin have the same stability.
The next result is a starting point for this analysis.

\begin{lema}\label{lema characterization origin}
In the Poincar\'e compactification, equation \eqref{main equation} satisfies:
\begin{enumerate}
\item  there are no equilibria at infinity
if and only if  $|s_2|>1;$
\item  when $|s_2|>1$,  infinity is an attractor when $s_1s_2>0$ and a repeller
when $s_1s_2<0$.
\end{enumerate}
\end{lema}

\begin{proof}
The proof is similar to that of Lemma $2.2$ in \cite{Rafel1}.
Using the change of variable $R=1/r$  in  \eqref{polar1}
and reparametrising time by
$\displaystyle{\frac{dt}{ds}=R},$ we obtain
$$
\left\{
\begin{array}{l}
R'=\frac{dR}{ds}=-2R\left(s_1-\cos(2n\theta)\right)-2p_1R^2\\
\\
\theta'= \frac{d\theta}{ds}=s_2+\sin(2n\theta)+p_2R .
\end{array}
\right.
$$
The invariant set $\{R=0\}$  corresponds to  infinity in \eqref{polar1}.
Hence, there are no equilibria at infinity
if and only if $|s_2|>1$.
The stability of infinity in
this case (see \cite{Lloyd})  is given by the sign of
\begin{equation*}
\displaystyle{\int_0^{2\pi}\frac{-2(s_1-\cos(2n\theta))}{s_2+\sin(2n\theta)}d\theta}=\displaystyle{\frac{-\hbox{sgn}(s_2)4\pi s_1}{\sqrt{s_2^2-1}}},
\end{equation*}
 and the result follows.
 Note that since we are assuming $|s_2|>1$, the integral above is always well defined.
\end{proof}

\section{Analysis of  equilibria}\label{Analysis of the equilibria}
In this section we describe the number of equilibria of \eqref{main equation}.
We start with the origin, that is  an equilibrium for all values of the parameters.
First we show that there is no trajectory of the differential equations that approaches the origin with a definite limit direction:
 the origin is \emph{monodromic}


\begin{lema}\label{lema Lyapunov}
If $p_2\ne 0$ then the origin  is a  monodromic equilibrium of \eqref{main equation}.
It is unstable if $p_1>0$, asymptotically stable if $p_1<0$.
If $p_1=0$ it is unstable if $s_1>1$, asymptotically stable if $s_1<-1$.
\end{lema}

Note that if $p_1=s_1=0$, equation \eqref{main equation} is Hamiltonian.
In this case the origin is a centre.
In Section~\ref{SecAbel} below we obtain better estimates for the case $p_1=0$.

\begin{proof}
To show that the origin is monodromic we compute the arriving directions of the flow to the origin, see \cite[Chapter IX]{Andronov} for  details.
We look for solutions that arrive at the origin tangent to a direction $\alpha$ that are zeros of $r^2\dot{\alpha}=R(x,y)=yP(x,y)+xQ(x,y)$.
If the   term of lowest degree in the polynomial $R(x,y)=-yP(x,y)+xQ(x,y)$ has no real roots then the origin is monodromic.

The term of lowest degree in  \eqref{main equation} is
$\pp z^{n-1}\bar{z}^{n-2}=P(x,y)+iQ(x,y)$, with real part
\begin{equation}\label{lowest4}
\begin{array}{l}
P(x,y)=-p_2y(x^2+y^2)^{n-2}+p_1x(x^2+y^2)^{n-2}
\end{array}
\end{equation}
and imaginary part
\begin{equation}\label{lowest5}
\begin{array}{l}
Q(x,y)=p_2x(x^2+y^2)^{n-2}+p_1y(x^2+y^2)^{n-2}.
\end{array}
\end{equation}

Then we have that $R(x,y)=-yP(x,y)+xQ(x,y)$ is given by
\begin{equation}\label{lowest3}
\begin{array}{l}
R(x,y)=p_2(x^2+y^2)(x^2+y^2)^{n-2}=p_2(x^2+y^2)^{n-1}
\end{array}
\end{equation}
which has no nontrivial real roots if $p_2\ne 0$, so the origin is monodromic.

From the expression for $\dot r$ in \eqref{polar1} it follows that if $p_1>0$ then for $r$ close to $0$, we have $\dot r>0$, hence the origin is unstable. Similarly,  if $p_1<0$ then  $\dot r<0$ for $r$ close to $0$, and the origin is asymptotically stable.
When $p_1=0$ the expression for $\dot r$ is $\dot r=2r^2\left(s_1-\cos(2n\theta)\right)$,  the origin is unstable if $s_1>1$, stable if $s_1<-1$.
\end{proof}

We now look for conditions under which  \eqref{main equation} has nontrivial equilibria.
We use equation \eqref{polarn} with the variable $\phi=n\theta$ to obtain simpler expressions, and analyse two open sets that cover the fundamental domain $0\le \phi<\pi$.

\begin{lema}\label{lema solutions r theta}
If  $|s_2|>1$ and $p_1\ne 0$ then equilibria of  equation \eqref{polarn} with $r>0$ exist if and only if $\Delta=p_1^2+p_2^2-\left(p_1s_2-p_2s_1\right)^2 \ge 0$.\\
If  $T_+=p_2-p_1s_2+p_2s_1\ne 0$ then equilibria of   \eqref{polarn}  with $-\pi/2<\phi<\pi/2$  satisfy:
\begin{equation}\label{solution r theta }
r_\pm=\frac{-p_2}{s_2+\sin\left(2\phi_{\pm}\right)},
\qquad\qquad
\tan\left(\phi_{\pm}\right)=
\frac{p_1\pm \sqrt{\Delta}}{p_2-p_1s_2+p_2s_1} .
\end{equation}
For equilibria with $0<\phi<\pi$  and $r\ne 0$ the restriction is  $T_-=p_2+p_1s_2-p_2s_1\ne 0$ and they satisfy
\begin{equation}\label{solution r theta2}
r_\pm=\frac{-p_2}{s_2+\sin\left(2\phi_{\pm}\right)},
\qquad\qquad
\cot\left(\phi_{\pm}\right)=
\frac{p_1\pm \sqrt{\Delta}}{-(p_2+p_1s_2-p_2s_1)} .
\end{equation}
There is only one equilibrium of  \eqref{polarn} with $r\ne 0$  and $-\pi/2<\phi<\pi/2$ when $T_+= 0$,
and it satisfies $\tan(\phi/2)=T_-/2p_1$. Similarly, for $T_-=0$, there is only one nontrivial equilibrium   with $0<\phi<\pi$, with  $\cot(\phi/2)=T_+/2p_1$.
\end{lema}

\begin{proof}
Let  $-\pi/2<\phi<\pi/2$.
The equilibria of  \eqref{polarn}, are the solutions of:\begin{equation}\label{equilibria0}
    \begin{array}{l}
    \displaystyle{0=2rp_1+2r^2\left(s_1-\cos (2\phi)\right)}\\
    \displaystyle{0=p_2+r\left(s_2+\sin(2 \phi)\right).}
    \end{array}
\end{equation}

For  $t=\tan(\phi)$, we have
\begin{equation}\label{Trigo}
\sin (2\phi)={\frac{2t}{1+t^2}}\qquad
\cos (2\phi)={\frac{1-t^2}{1+t^2}}.
\end{equation}
Since $s_2+\sin\phi\ne 0$, we may eliminate $r=-p_2/\left(s_2+\sin(2\phi)\right)$ from equations \eqref{equilibria0}  to get
\begin{equation}\label{grau2}
(-p_2+p_1s_2-p_2s_1)t^2+2p_1t+p_2+p_1s_2-p_2s_1=0,
\end{equation}
or, equivalently, $T_+ t^2-2p_1t+T_-=0$.
If the coefficient $T_+$ of $t^2$ is zero, then equation \eqref{grau2} is linear in $t$ and hence has only one solution, $t=T_-/2p_1$.
When the coefficient of $t^2$ is not zero, solving  equation \eqref{grau2} for $t$  yields the result.

Finally, consider the interval $0<\phi<\pi$ and let $\tau=\cot\phi$.
The expression \eqref{Trigo} for $\sin 2\phi$ is the same, and that of  $\cos 2\phi$ is multiplied by -1.
Instead of \eqref{grau2} we get $T_- t^2-2p_1t+T_+=0$ and the result follows by the same arguments.
\end{proof}

Note that when both $T_+$ and $T_-$ are not zero, the expressions for $\tan(\phi/2)$ and $\cot(\phi/2)$ above define the same angles, since
$$
(p_1\pm\sqrt{\Delta})/T_+=-T_-/(p_1\mp\sqrt{\Delta}).
$$

In Lemma~\ref{lema solutions r theta} we found the  number of equilibria  of equation \eqref{polar1}   with $r\ne 0$
in the regions  $-\pi/2<\phi<\pi/2$ and $0<\phi<\pi$.
To complete the information it remains to  deal with the case when, for the same value of the parameters, two equilibria may occur,  each  in one  of these intervals but not in the other.

\begin{lema}\label{lema0pi}
If  $|s_2|>1$ and $\pp\ne 0$, there are no parameter values for which \eqref{polarn}  has equilibria with $r> 0$ simultaneously for $\phi=0$ and $\phi=\pi/2$.
\end{lema}
\begin{proof}
If we solve
$$
\left\{
\begin{array}{l}
0=p_1+r(s_1-\cos(2 \phi))\\
0=p_2+r(s_2+\sin(2 \phi))
\end{array}
\right.
$$
for $\phi=0$, we get  a solution  at $r=-p_2/s_2$ subject to the condition $T_-=p_2+p_1s_2-p_2s_1=0$.
Solving the same system for $\phi=\pi/2$ yields an equilibrium at $r=-p_2/s_2$ under the restriction $T_+=p_2-p_1s_2+p_2s_1=0$.
The parameter restrictions for $\phi=0$ and $\phi=\pi/2$ are equivalent to  $p_1s_2=p_2(s_1-1)$
and $p_1s_2=p_2(s_1+1)$, respectively.
Hence, in order to have equilibria at $\phi=0$ and $\phi=\pi/2$  for the same parameters,
it is necessary to have $p_1=p_2=0$.
\end{proof}

In the following we summarize the conditions that the parameters have to fulfill in order that  \eqref{polar1} has exactly
one, $2n+1$ or $4n+1$ equilibria  (see Figure~\ref{BD}).

\begin{prop}\label{7equilibria}
For $|s_2|>1$  and $\pp\ne 0$,  if $p_2s_2\ge 0$,  then the only equilibrium of \eqref{main equation}  is the origin.
If $p_2s_2<0$ then the number of equilibria of  \eqref{main equation} is determined by the quadratic form
$\QQ (p_1,p_2)$ defined in \eqref{quadraticForm}
and is:
\begin{enumerate}
\item \label{oneEq}
exactly one equilibrium (the origin) if $\QQ (p_1,p_2)<0$;
\item \label{sevenEq}
exactly $2n+1$ equilibria    if $\QQ (p_1,p_2)=0$;
\item \label{thirteenEq}
exactly $4n+1$ equilibria   if $\QQ (p_1,p_2)>0$;
\end{enumerate}
\end{prop}

Since  $\QQ $ is a quadratic form on $p_1,p_2$, and since  its determinant $1-s_1^2-s_2^2,$  is negative when
$|s_2|>1$, then for each choice of $s_1, s_2$ with $s_2>1$, the
points where $\QQ (p_1,p_2)$ is positive lie on two sectors, delimited by the two  lines where $\QQ (p_1,p_2)=0$.
Also $\QQ (p_1,0)=(1-s_2^2)p_1^2<0$ for $|s_2|>1$, and thus
the sectors where there are two equilibria in each $\theta=\pi/n$ do not
include the $p_1$ axis, as in Figure~\ref{BD}.


\begin{proof}
Due to the $\ZZ_{2n}$-symmetry, the number of nontrivial equilibria of \eqref{main equation} will be $2n$ times the number of equilibria of \eqref{polar1} with $r>0$ and $\theta\in\left[0,\pi/n\right)$, or equivalently, $2n$ times  the number of equilibria of \eqref{polarn} with $r>0$ and $\phi\in\left[0,\pi\right)$.

If $p_2$ and $s_2$ have the same sign, then the expression for $r$ in \eqref{solution r theta2} is negative, and there are no solutions with $0<\phi<\pi$. For $\phi=0$, Lemma~\ref{lema0pi} gives the value $-p_2/s_2$ for $r$, that would also be negative, so there are no nontrivial equilibria if $p_2s_2>0$.

Suppose now $p_2s_2<0$.
By Lemma~\ref{lema solutions r theta}, there are no solutions $\phi$ when the discriminant $\Delta$ is negative, corresponding to $\QQ (p_1,p_2)<0$ as in assertion {\sl(\ref{oneEq})}.
Ignoring for the moment the restriction $R_-\ne 0$, there are  exactly $2$ solutions $\phi\in\left(0,\pi\right)$ if $\Delta>0$, that corresponds to $\QQ (p_1,p_2)<0$, and this gives us assertion {\sl(\ref{thirteenEq})}.

In order to have exactly $2n$ nontrivial equilibria, two conditions have to be satisfied: $p_2s_2<0$ to ensure positive values
of $r$, and the quantities $\cot(\phi_{\pm})$ have to coincide, {\it i.e.}
the discriminant $\Delta$ in Lemma~\ref{lema solutions r theta} has to be zero, hence, $r_+=r_-$ and $\phi_+=\phi_-$,
assertion {\sl (\ref{sevenEq})}  in the statement.

Finally, if $R_-=0$, Lemma~\ref{lema solutions r theta} provides only one solution $\phi\in\left[0,\pi\right)$,
but  in this case $\phi=0$ is also a solution, by Lemma~\ref{lema0pi}.
When $R_-=0$ we have $\QQ (p_1,p_2)=p_1^2>0$, so we are in the situation of assertion  {\sl(\ref{oneEq})} if $p_2s_2<0$.
\end{proof}

\begin{lema}\label{lemaStability}
For $|s_2|>1$, and $p_2s_2< 0$,
 if  $\QQ (p_1,p_2)=0$ all the nontrivial equilibria  of \eqref{main equation} are saddle-nodes.
\end{lema}

\begin{proof}
The Jacobian matrix of  \eqref{polar1} is
\begin{equation}\label{jacobian general}
J_{(r,\theta)}=
\left(\begin{matrix}
            2p_1+4r(s_1-\cos(2n\theta))&4nr^2\sin(2n\theta)\\
            s_2+\sin(2n\theta)&2nr\cos(2n\theta)\\
        \end{matrix}\right).
\end{equation}

If  $\QQ (p_1,p_2)=0$ there is only one nontrivial equilibrium
with $-\pi/n<\theta\le\pi/n$, that we denote by $(r_+,\theta_+)$.
Substituting the expression  \eqref{solution r theta } into  the Jacobian matrix \eqref{jacobian general}  and taking into account that  $\Delta=\QQ (p_1,p_2)=0$,  the eigenvalues of the matrix are
\begin{equation}\label{sn}
\lambda_1=0
\quad
\lambda_2=2p_1-2p_2 \frac{\left(2s_1-n+2\right)T_+^2+p_1^2\left(2s_1+n-2\right)}{s_2T_+^2+2p_1T_+ + p_1s_2},
\end{equation}
where $T_+$ was defined in Lemma~\ref{lema solutions r theta}.

Therefore $(r_+,\theta_+)$ has a zero eigenvalue, and the same holds for its $2n$ copies by the symmetry.
To show that these equilibria are saddle-nodes  we use the  well--known fact  that the sum of the indices of all equilibria contained in the interior of a limit cycle of a planar system
is $+1$ --- see  for instance \cite{Andronov}.
Since we are assuming $|s_2|>1$, by Lemma~\ref{lema characterization origin}, there are no equilibria at  infinity.
Hence, infinity is a limit cycle of the system and it has $2n+1$ equilibria in its interior: the origin, that is a focus and hence has index +1,
and $2n$ other equilibria, all of the same type because of the symmetry. Consequently, the index of these equilibria must be 0. As we have proved that they are semi-hyperbolic equilibria then they must be saddle-nodes.
\end{proof}

This completes the proof of Theorem~\ref{teoEquilibria}.

%
%

\section{Reduction to the Abel equation}\label{SecAbel}
In this section we start to address the existence of limit cycles for equation \eqref{main equation}.

\begin{lema}\label{lema abelian}
The  periodic orbits  of equation \eqref{main equation}   that surround the origin are in one-to-one correspondence with
the
non contractible solutions that satisfy
$x(0)=x(2\pi)$ of the Abel equation
\begin{equation}\label{abelian0}
    \begin{array}{l}
    \displaystyle{\frac{dx}{d\theta}=A(\theta)x^3+B(\theta)x^2+C(\theta)x}
    \end{array}
\end{equation}
where
\begin{equation}\label{abelian1}
    \begin{array}{l}
    \displaystyle{A(\theta)=\frac{2}{p_2}\left(p_1+p_1s_2^2-p_2s_1s_2+(-p_2s_1+2p_1s_2)\sin(2n\theta)\right)+}\\
    \hspace{1.3cm}+\displaystyle{\frac{2}{p_2}\left(-p_1\cos(2n\theta)+p_2s_2+p_2\sin(2n\theta))\cos(2n\theta\right),}\\
    \\
    \displaystyle{B(\theta)=\frac{2}{p_2}\left(-2p_1s_2+p_2s_1-2p_1\sin(2n\theta)-p_2\cos(2n\theta) \right),}\\
    \\
    \displaystyle{C(\theta)=\frac{2p_1}{p_2}.}\\
    \\
    \end{array}
\end{equation}
\end{lema}

\begin{proof}
From equation \eqref{polar1} we obtain
\begin{equation*}
\frac{dr}{d\theta}=\frac{2r(p_1+r\left(s_1+\cos (2n\theta)\right))}{p_2+r\left(s_2+\sin (2n\theta)\right)}.
\end{equation*}
Applying the Cherkas transformation $x=\displaystyle{\frac{r}{p_2+r\left(s_2+\sin (2n\theta)\right)}}$, see \cite{Cherkas},
we get the scalar equation \eqref{abelian0}. The limit cycles that surround the origin of equation \eqref{main equation} are transformed
into non contractible periodic orbits of
equation \eqref{abelian0}, as they cannot intersect the set $\{\dot{\theta}=0\}$, where the denominators of $dr/d\theta$ and of the Cherkas transformation vanish.
 For more details see \cite{CGP}.
\end{proof}

\begin{coro}\label{lemaStabLyapunov}
If $p_2\ne 0$ and $p_1=0$ then the origin  is an asymptotically stable  equilibrium of \eqref{main equation} if $s_1<0$,
unstable if $s_1>0$, .
\end{coro}

\begin{proof}
The stability of the origin can be determined from the two first  Lyapunov constants.
For an
Abel equation they are given by$$
V_1=\displaystyle{\exp\left(\int_0^{2\pi}C(\theta)d\theta\right)}-1,\qquad\qquad
V_2=\displaystyle{ \int_0^{ 2\pi }B(\theta)d\theta}.
$$
Using this, we get from the expressions given in \eqref{abelian1} that
if $p_1=0$ then $C(\theta)=0$ implying $V_1=0$.
On the other hand
%
$
V_2=4\pi s_1,
$
and we get the result.
\end{proof}

\begin{lema}\label{signA}
For $|s_2|>1$ the function $A(\theta)$ of Lemma \ref{lema abelian} changes sign if and only if
$\QQ(p_1,p_2)>0$, where $\QQ$ is the quadratic form defined in \eqref{quadraticForm}.
\end{lema}

\begin{proof}
Writing $x=\sin(2n\theta),~y=\cos(2n\theta)$, the function $A(\theta)$ in \eqref{abelian1} becomes
\begin{equation*}
A(x,y)=\displaystyle{\frac{2}{p_2}\left(p_1-p_2s_1s_2+p_1s_2^2+(2p_1s_2-p_2s_1)x+  (p_2x-p_1y+p_2s_2)y\right)},
 \end{equation*}
and we solve the set of equations
\begin{eqnarray*}
 A(x,y)=0,\\
x^2+y^2=1.
\end{eqnarray*}
to get the solutions
\begin{equation*}
\begin{array}{l}
x_1=-s_2,~y_1=\sqrt{1-s_2^2}\\
x_2=-s_2,~y_2=-\sqrt{1-s_2^2}\\\\
x_\pm=\displaystyle{\frac{p_1p_2s_1-p_1^2s_2\pm p_2\sqrt{\QQ(p_1,p_2)}}{p_1^2+p_2^2}}\\
\\
y_\pm=\displaystyle{\frac{p_2^2s_1-p_1p_2s_2\mp p_1 \sqrt{\QQ(p_1,p_2)}}{p_1^2+p_2^2}}\\
\end{array}
\end{equation*}
The  first two pairs of solutions $(x_1,y_1), (x_2,y_2)$ cannot be solutions of $A(\theta)=0$ since $x=\sin(2n\theta)=-s_2$ and we are assuming $|s_2|>1$.

If we look for the intervals where the expression $A(x,y)$ does not change sign we have two possibilities:
either $|x_\pm|>1$ (again not compatible with $x=\sin(2n\theta)$ nor with $x^2+y^2=1$) or  the discriminant  $\QQ(p_1,p_2)$ is negative or zero. In the  case $\QQ(p_1,p_2)<0$, there will be no real solutions $x$, $y$.
If $\QQ(p_1,p_2)=0$, the function $A(\theta)$ will have a double zero  and will not
change sign.
So the only possibility is to have $\QQ(p_1,p_2)>0$.
\end{proof}

\begin{lema}\label{sigmalemaB}
For $|s_2|>1$ the function $B(\theta)$ of Lemma \ref{lema abelian} changes sign if and only if
$\QQ(2p_1,p_2)>0$, where $\QQ$ is the quadratic form defined in \eqref{quadraticForm}.
\end{lema}

\begin{proof}
Using the substitution of the proof of the previous lemma, $x=\sin(2n\theta), y=\cos(2n\theta),$
we get that the solutions of the system
\begin{eqnarray*}
B(x,y)=0,\\
x^2+y^2=1,
\end{eqnarray*}
are
\begin{eqnarray*}
x_\pm= \frac{2p_1p_2s_1-4p_1^2s_2\pm p_2\sqrt{\QQ(2p_1,p_2)}}{4 p_1^2+p_2^2},\\
y_{\pm}=\frac{p_2^2s_1-2p_1p_2s_2\mp 2p_1\sqrt{\QQ(2p_1,p_2)}}{4 p_1^2+p_2^2}.
\end{eqnarray*}
By the same arguments of the previous proof,  we get that the function $B(\theta)$ will not change sign if and only if
$\QQ(2p_1,p_2)\le 0$
\end{proof}

In order to complete the proof of Theorem~\ref{teorema principal} in this  section, we will need some results on Abel equations proved in \cite{pliss} and \cite{Llibre}, that we summarise in a theorem.

\begin{teo}[Pliss 1966, Gasull \&{} Llibre 1990]\label{teorema Llibre}
Consider the Abel equation \eqref{abelian0} and assume that either $A(\theta)\not\equiv0$ or $B(\theta)\not\equiv0$ does not change sign.
Then it has at most three solutions satisfying $x(0)=x(2\pi),$ taking into account their multiplicities.
\end{teo}

\section{Analysis of limit cycles}\label{section proof}

\begin{proof}[Proof of Theorem~\ref{teorema principal}]
For assertion $(a)$, define the function $c(\theta)$ by  $c(\theta)=s_2+\sin(2n\theta)$.
Since $|s_2|>1$, we have $c(\theta)\neq0,~\forall\theta\in[0,2\pi]$ and
a simple calculation shows that the curve $x=1/c(\theta)$ is a solution of  \eqref{abelian0} satisfying
$x(0)=x(2\pi)$.
As shown in \cite{Rafel1}, doing the Cherkas transformation backwards we get that $x=1/c(\theta)$ is mapped into infinity of the original differential equation.

Assume  that one of conditions $(i)$ or $(ii)$ is satisfied.
By Lemma~\ref{lema abelian}, we reduce the study of the periodic orbits of
equation \eqref{main equation} to the analysis of the  non contractible periodic orbits of the Abel equation \eqref{abelian0}.
If $\QQ(p_1,p_2)\le 0$,  by  Lemma~\ref{signA}, the function $A(\theta)$
in the Abel equation does not change sign.
If $\QQ(2p_1,p_2)\le 0$ then $B(\theta)$ does not change sign, by Lemma~\ref{sigmalemaB}.
In both cases, Theorem~\ref{teorema Llibre}
ensures that there are at most three solutions, counted with multiplicities, of  \eqref{abelian0} satisfying
$x(0)=x(2\pi)$. One of them is trivially $x=0$.
A second one is $x=1/c(\theta)$.
Hence, there is at most one more contractible solution of  \eqref{abelian0}, and by Theorem~\ref{teorema Llibre},
the maximum number of limit cycles of equation \eqref{main equation} is one. 
Moreover, from the same theorem
it follows that when the limit cycle exists it
has multiplicity one and hence it 
 is hyperbolic.
This completes the proof of  assertion $(a)$ in Theorem~\ref{teorema principal}.

For assertion $(b)$, let $s_2>1$, $s_1<0$ and choose $p_1>0$ and $p_2\ne 0$  in the region $\QQ(p_1,p_2)<0$
(for instance, $s_1=-1/2$, $s_2=2$, $p_1=p_2=1$, $\QQ(p_1,p_2)=-17/4<0$).
By Theorem~\ref{teoEquilibria} the only equilibrium is the origin, and by Lemma~\ref{lema Lyapunov} it is a repeller since $p_2\ne 0$ and $p_1>0$.
Infinity is also a repeller by Lemma~\ref{lema characterization origin}, because $s_1s_2<0$.
By the Poincar\'e-Bendixson Theorem and by the first part  of the proof of this theorem, there is exactly one hyperbolic limit cycle surrounding the origin.
Moreover, this  limit cycle  is  stable.
An unstable limit cycle may be obtained changing the signs of $p_1$, $s_1$ and $s_2$.

For assertion $(c)$, we  take $s_2>1$, $s_1<0$ and choose $p_1>0$ and $p_2< 0$  in one of the lines $\QQ(p_1,p_2)=0$
(for instance, $s_1=-1/2$, $s_2=2$, $p_1=(2+\sqrt{13})/6$, $p_2=-1$, $\QQ(p_1,p_2)=0$).
By the same arguments above,  both the origin and infinity are repellers, since $s_1s_2<0$, $p_2\ne 0$ and $p_1>0$.
Also, since $p_2s_2<0$,  by Theorem~\ref{teoEquilibria} there is exactly one equilibrium,  a saddle-node,  in each region  $(k-1)\pi/n\le \theta<k\pi/n$, $k\in\ZZ$.
Again by $(a)$ there is at most one limit cycle.
In order to show that this cycle exists and encircles the saddle-nodes, we will construct a polygonal line from the origin to the saddle-node $(r_*,\theta_*)$, $-\pi/n<\theta_*<0$  where the vector field points outwards, away from the saddle-node, see Figure~\ref{figure polygon}.
Copies of the poligonal by the symmetries will join the origin to the other saddle-nodes and the union of all these will form a polygon where the vector field points outwards, away from the saddle-nodes.
Since infinity is a repeller and there are no equilibria outside the polygon, by the Poincar\'e-Bendixson Theorem there will be a limit cycle encircling the saddle-nodes.

For the construction of the polygon we need some information on the location of the saddle-node $z_*=(r_*,\theta_*)$.
Solving $\QQ(p_1,p_2)=0$ for $p_1$  yields
$$
p_1=-p_2\frac{s_1s_2\pm \sqrt{s_1^2+s_2^2-1}}{1-s2^2}.
$$
Choosing the solution with the minus sign and substituting into  \eqref{solution r theta }, we get
$$
\frac{1}{\tan(n\theta_*)}=\frac{p_2}{p_1}(1+s_1)-s_2=
\frac{(s_2^2-1)(1+s_1)}{s_1s_2-\sqrt{s_1^2+s_2^2-1}}-s_2<-1 .
$$
Therefore $-1<\tan(n\theta_*)<0$ and hence $-\pi/4n<\theta_*<0$.

For the first piece of the polygonal we look at the ray $\theta=-\pi/4n$, where
$\dot\theta<0$ if $0<r<r_0=-p_2/(s_2+\sqrt{2}/2)$.
Therefore on the segment $0<r<r_0$ the vector field points away from the saddle-node $z_*$.

Another piece of the polygonal will be contained in the line $z_*+xv$ where $x\in\RR$ and  $v$ is an eigenvector corresponding to the non-zero eigenvalue of $z_*$.
This line is tangent to the separatrix of the saddle region of the saddle-node.
Let $x_0$ be the smallest positive value of $x$ for which the vector field is not transverse to this line.

If the ray intersects the tangent to the separatrix at a point with $0<r<r_0$ and with $0<x<x_0$, then the polygonal is the union of the two segments, from the origin to the intersection and from there to $z_*$.
Otherwise, the segment joining the point $z_1=z_*+x_0 v$ to the point $z_2$ with $r=r_0$, $\theta=-\pi/4n$ will also be transverse to the vector field, and the polygonal will consist of the three segments from the origin to $z_2$, from there to $z_1$, and whence to $z_*$. This completes the construction of the polygonal, and hence, the proof of assertion $(c)$.

\begin{figure}
\begin{center}
\includegraphics[scale=0.5]{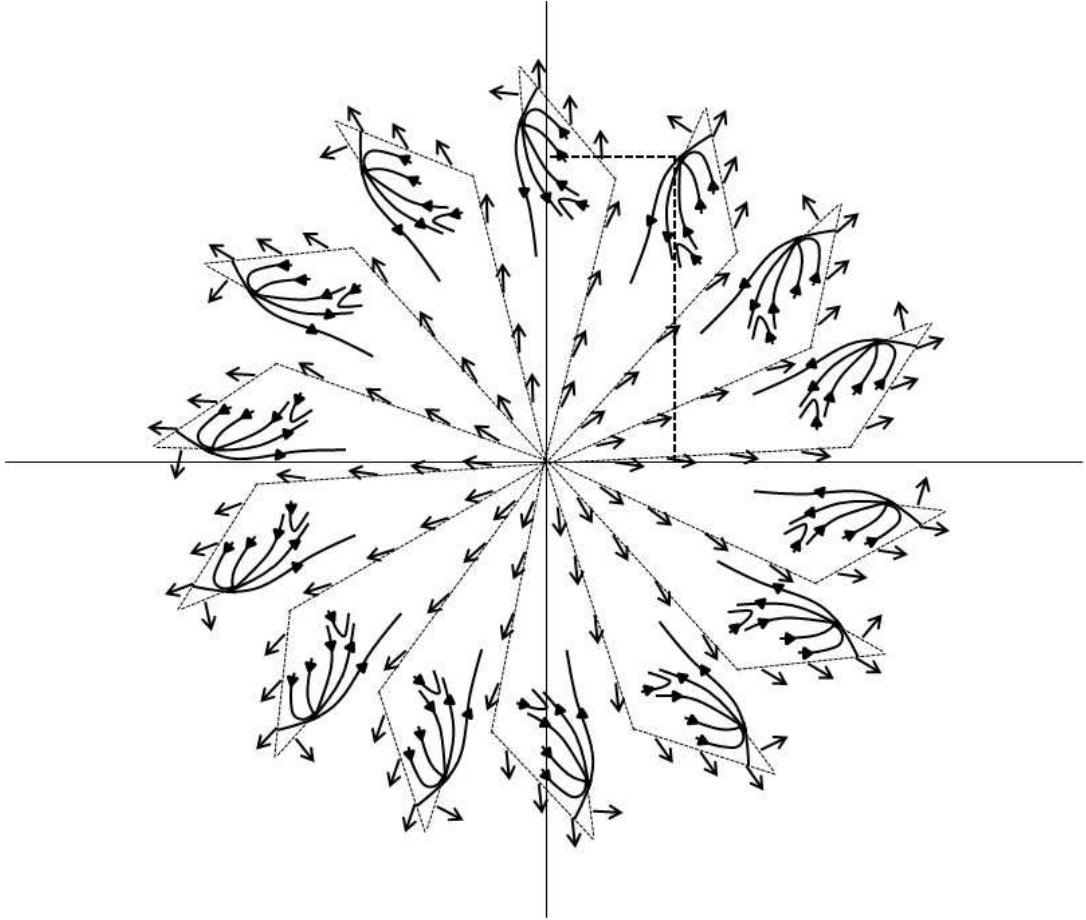}
\caption{The polygonal curve transverse to the flow of the differential equation and the separatrices of the saddle-nodes of  \eqref{main equation}, for $n=7$.}\label{figure polygon}
\end{center}
\end{figure}

Finally, for assertion $(d)$ we start with parameters for which $(c)$ holds with $\QQ(2p_1,p_2)<0$.
The example given above, $s_1=-1/2$, $s_2=2$, $p_1=(2+\sqrt{13})/6$, $p_2=-1$, $\QQ(p_1,p_2)=0$ satisfies
$\QQ(2p_1,p2)=-\left(13+8\sqrt{13}\right)/12<0$.
By Lemma~\ref{sigmalemaB}, the function $B(\theta)$ does not change sign.
The hyperbolic limit cycle persists under small changes of parameters, and  $\QQ(2p_1,p_2)$ is still negative, while  moving the parameters away  from the line   $\QQ(p_1,p_2)=0$.
When the parameters move into the region where  $\QQ(p_1,p_2)>0$, each saddle-node splits into two equilibria that are still encircled by the limit cycle. Moving in the opposite direction, int  $\QQ(p_1,p_2)<0$ destroys all the non-trivial equilibria, and only the origin remains inside the limit cycle. Thus, all situations of assertion $(d)$ occur.
\end{proof}

\end{document}